\documentclass[11pt,reqno,a4paper]{amsart}
\usepackage[utf8]{inputenc}
\usepackage[T1]{fontenc}
\usepackage{lmodern}
\usepackage{amsmath}
\usepackage{amssymb}
\usepackage{amsfonts}
\usepackage{mathrsfs}
\usepackage{amsthm}
\usepackage{commath}
\usepackage{hyperref}
\usepackage{doi}
\usepackage{bigints}
\usepackage[dvipsnames]{xcolor}
\usepackage{dsfont}
\usepackage{graphicx}
\usepackage{enumitem}
\usepackage{import}
\usepackage{chngcntr}

\hypersetup{
	colorlinks   = true,
	urlcolor     = Blue,
	linkcolor    = Blue,
	citecolor   = Maroon,
	breaklinks} 

\newcommand{\cB}{\mathcal{B}}
\newcommand{\R}{\mathbb{R}}
\newcommand{\Z}{\mathbb{Z}}

\newcommand{\E}{\mathcal{E}_{\mu,p}}

\DeclareMathOperator{\supp}{supp}
\DeclareMathOperator{\graph}{graph}
\DeclareMathOperator{\lip}{Lip}
\newcommand\restr[2]{{
		\left.\kern-\nulldelimiterspace 
		#1 
		\right|_{#2} 
}}

\newcommand{\Hr}[1]{\restr{\mathcal{H}^n}{#1}}

\newcommand{\Hn}{\mathcal{H}^n}
\newcommand{\Ho}{\mathcal{H}^1}

\DeclareMathOperator{\diam}{diam}
\DeclareMathOperator{\dist}{dist}

\newtheorem{theorem}{Theorem}[section]
\newtheorem*{theorem*}{Theorem}
\newtheorem{lemma}[theorem]{Lemma}

\newtheorem{prop}[theorem]{Proposition}
\theoremstyle{definition}
\newtheorem{definition}[theorem]{Definition}
\newtheorem*{definition*}{Definition}
\theoremstyle{remark}

\newtheorem{question}[theorem]{Question}

\numberwithin{equation}{section}
\counterwithin{figure}{section}

\title{Two examples related to conical energies}
\author[D. D\k{a}browski]{Damian D\k{a}browski}
\address{Damian D\k{a}browski\\ Departament de Matem\`atiques, Universitat Aut\`onoma de Barcelona, and Barcelona Graduate School of Mathematics (BGSMath)\newline\indent Edifici C Facultat de Ci\`encies, 08193 Bellaterra, Barcelona, Catalonia, Spain\newline\indent
	Current address: P.O. Box 35 (MaD), 40014 University of Jyväskylä, Finland}
\email{damian.m.dabrowski ``at'' jyu.fi}

\begin{document}
		\begin{abstract}
		In a recent article \cite{dabrowski2020cones} we introduced and studied \emph{conical energies}. We used them to prove three results: a characterization of rectifiable measures, a characterization of sets with big pieces of Lipschitz graphs, and a sufficient condition for boundedness of nice singular integral operators. In this note we give two examples related to sharpness of these results. One of them is due to Joyce and M\"{o}rters \cite{joyce2000set}, the other is new and could be of independent interest as an example of a relatively ugly set containing big pieces of Lipschitz graphs.
	\end{abstract}
	\maketitle
	\section{Introduction}
	In \cite{dabrowski2020cones} we introduced \emph{conical energies}. Let us recall their definition. Given $x\in\R^d,\,\alpha\in(0,\pi/2)\,$ and an $m$-dimensional plane $V\in G(d,m)$, set
	\begin{equation*}
		K(x,V,\alpha) = \{y\in\R^d\, :\, \dist(y,V+x) < \sin(\alpha)|y-x| \}.
	\end{equation*}
	In other words, $K(x,V,\alpha)$ denotes the open cone centered at $x$ with direction $V$ and aperture $\alpha$.\footnote{In \cite{dabrowski2020cones} we defined cones slightly differently, with ``$\sin(\alpha)$'' replaced by ``$\alpha$''. Obviously one can pass between the two definitions easily, and the current definition is more natural in $\R^2$.} The truncated cone $K(x,V,\alpha)\cap B(x,r)$ will be denoted by $K(x,V,\alpha,r)$.
	\begin{definition}
		Suppose $\mu$ is a Radon measure on $\R^d$, and $x\in\supp\mu$. Let $V\in G(d,d-n),\ \alpha\in(0,\pi/2),\ 1\le p <\infty$ and $R>0$. We define the \emph{$(V,\alpha,p)$-conical energy of $\mu$ at $x$ up to scale $R$} as
		\begin{equation*}
			\E(x,V,\alpha,R) = \int_0^{R}\bigg(\frac{\mu(K(x,V,\alpha,r))}{r^n} \bigg)^p\ \frac{dr}{r}.
		\end{equation*}
		For $E\subset\R^d$ we set also $\mathcal{E}_{E,p}(x,V,\alpha,R)=\mathcal{E}_{\Hr{E},p}(x,V,\alpha,R)$.
	\end{definition}
	Note that the definition above depends on the dimension parameter $n$, so to be more precise one could say that $\E(x,V,\alpha,R)$ is the $n$-dimensional $(V,\alpha,p)$-conical energy. For the sake of simplicity, for the rest of the introduction we will consider $n$ to be fixed, and we will not point out this dependence. The same applies to other definitions. Throughout the paper we will actually work with $n=1$.

	For $p=1$ the conical energies were first considered in \cite{chang2017analytic} where the authors used them to prove an inequality involving analytic capacity and projections. For $p>1$ they were defined in \cite{dabrowski2020cones}. A related quantity was also independently introduced in \cite{badger2020radon} -- Badger and Naples used their \emph{conical defect} to characterize measures concentrated on a countable union of Lipschitz graphs.
	
	In \cite{dabrowski2020cones} we used the conical energies to prove three results: a characterization of rectifiable measures, a characterization of sets with big pieces of Lipschitz graphs, and a sufficient condition for boundedness of nice singular integral operators. Below we briefly describe the last two theorems. The aim of this note is to give two examples related to sharpness of these results. For more information on conical energies, as well as a full presentation of results obtained in \cite{dabrowski2020cones}, we refer the reader to the original paper. 
	
	\subsection{Big pieces of Lipschitz graphs}
	We begin by recalling some definitions.
	\begin{definition}
		We say that a Radon measure $\mu$ is \emph{$n$-Ahlfors-David regular} (abbreviated as $n$-ADR) if there exists a constant $C>0$ such that for all $x\in E$ and $0<r<\diam(E)$
		\begin{equation*}
			C^{-1}\, r^n\le \mu(B(x,r))\le C\, r^n.
		\end{equation*}
		The constant $C$ will be referred to as the ADR constant of $\mu$. Furthermore, we say that an $\mathcal{H}^n$-measurable set $E$ is $n$-ADR if $\Hr{E}$ is $n$-ADR.
	\end{definition}
	
	\begin{definition}\label{def:BPLG}
		We say that an $n$-ADR set $E\subset\R^d$ has \emph{big pieces of Lipschitz graphs} (BPLG) if there exist constants $\kappa, L>0,$ such that the following holds.
		
		For all balls $B$ centered at $E$, $0<r(B)<\diam(E),$ there exists an $n$-dimensional Lipschitz graph $\Gamma$ with $\text{Lip}(\Gamma)\le L$, such that
		\begin{equation*}\label{eq:kappa big piece}
			\Hn(E\cap B\cap \Gamma)\ge \kappa\, r(B)^n.
		\end{equation*}
	\end{definition}
	Sets with BPLG were studied in \cite{david1991singular,david1993analysis} as a potential quantitative counterpart of rectifiability. Few characterizations of such sets are available. In \cite{david1993quantitative} they were characterised in terms of the \emph{big projections} property and the \emph{weak geometric lemma}, in \cite{martikainen2018characterising} -- using $L^2$ norms of projections, and very recently in \cite{orponen2020plenty} using the \emph{plenty of big projections} property. In \cite{dabrowski2020cones} we characterize the sets with BPLG using the conical energy. More precisely, we show that containing BPLG is equivalent to the following property.
	\begin{definition}
		Let $1\le p<\infty$. We say that a measure $\mu$ has \emph{big pieces of bounded energy for $p$}, abbreviated as BPBE($p$), if there exist constants $\alpha,\kappa, M_0>0$ such that the following holds. 
		
		For all balls $B$ centered at $\supp \mu$, $0<r(B)<\diam(\supp\mu),$ there exist a set $G_B\subset B$ with $\mu(G_B)\ge\kappa\, \mu(B)$, and a direction $V_B\in G(d,d-n)$, such that for all $x\in G_B$
		\begin{equation}\label{eq:BPLG condition in thm}
			\mathcal{E}_{\mu,p}(x,V_B,\alpha,r(B))=\int_0^{r(B)} \bigg(\frac{\mu(K(x,V_B,\alpha,r))}{r^n} \bigg)^{p}\ \frac{dr}{r}\le M_0.
		\end{equation}
	\end{definition}	
	\begin{theorem}[{\cite[Theorem 1.11]{dabrowski2020cones}}]\label{thm:suff BPLG}
		Let $1\le p <\infty$. Suppose $E\subset\R^d$ is $n$-ADR. Then $E$ has BPLG if and only if $\Hr{E}$ has BPBE($p$).
	\end{theorem}
	It seemed to us rather natural to consider also the following property.
\begin{definition}\label{def:BME}
	Let $1\le p<\infty$. We say that a measure $\mu$ has \emph{bounded mean energy for $p$} (BME($p$)) if there exist constants $\alpha, M_0>0$, and for every $x\in \supp\mu$ there exists a direction $V_x\in G(d,d-n)$, such that the following holds. 
	
	For all balls $B$ centered at $\supp \mu$, $0<r(B)<\diam(\supp\mu),$ we have 
	\begin{multline*}
		\int_{B}\mathcal{E}_{\mu,p}(x,V_x,\alpha,r(B))\ d\mu(x)\\
		= \int_{B}\int_0^{r(B)} \bigg(\frac{\mu(K(x,V_x,\alpha,r))}{r^n} \bigg)^p\ \frac{dr}{r}d\mu(x) \le M_0\, \mu(B).
	\end{multline*}
\end{definition}
In other words we require $\mu(K(x,V_x,\alpha,r))^pr^{-np} \frac{dr}{r}d\mu(x)$ to be a Carleson measure. This condition looks quite natural due to many similar characterizations of so-called uniformly rectifiable sets, e.g. the geometric lemma of \cite{david1991singular, david1993analysis} or the results from \cite{tolsa2008uniform, tolsa2012mass}. In this paper we won't need the definition of uniform rectifiability, but let us note that all sets with BPLG are uniformly rectifiable, and that the BPLG condition is strictly stronger than uniform rectifiability. 

It is easy to show, using the compactness of $G(d,d-n)$ and Chebyshev's inequality, that BME($p$) implies BPBE($p$). However, the reverse implication does not hold. In Section \ref{sec:BPLG example} we construct the appropriate example.
\begin{theorem}
	There exists a $1$-ADR set $E\subset\R^2$ that contains big pieces of $1$-Lipschitz graphs, but it does not satisfy the BME($p$) condition for any $p\ge 1$.
\end{theorem}

The problem with BME is the following. Contrary to the aforementioned characterizations of uniform rectifiability, in BME the ``approximating'' plane $V_x$ is fixed for every $x\in\supp\mu$ once and for all, and we do not allow it to change between different scales. As shown by our example, this is too rigid.
\begin{question}
	Suppose one modifies the definition of BME, allowing the planes $V_x$ to depend on $r$, perhaps with some additional control on the oscillation of $V_{x,r}$. Can the modified BME be used to characterize BPLG, or uniform rectifiability?
\end{question}

	\subsection{Boundedness of SIOs}
	We consider singular integral operators of convolution type, with odd $C^2$ kernels $k:\R^d\setminus\{0\}\to\R$ satisfying for some constant $C_k>0$
	\begin{equation}\label{eq:calderon zygmund constant}
		|\nabla^j k(x)|\le \frac{C_k}{|x|^{n+j}}\quad \text{for $x\neq 0$}\quad \text{and}\quad j\in\{0,1,2\}.
	\end{equation}
	We will denote the class of all such kernels by $\mathcal{K}^n(\R^d)$. Note that these kernels are particularly nice examples of Calder\'{o}n-Zygmund kernels.
	
	\begin{definition}\label{def:SIO}
		Given a kernel $k\in\mathcal{K}^n(\R^d)$, a constant $\varepsilon>0$, a Radon measure $\mu$ and a function $f\in L^1_{loc}(\mu)$ we define
		\begin{equation*}
			T_{\mu,\varepsilon}f(x) = \int_{|x-y|>\varepsilon}k(y-x)\, f(y)\, d\mu(y),\quad x\in\R^d.
		\end{equation*}
		We say that $T_{\mu}$ is bounded in $L^2(\mu)$ if all $T_{\mu,\varepsilon}$ are bounded in $L^2(\mu)$, uniformly in $\varepsilon>0$.
	\end{definition}
	In their seminal work \cite{david1991singular} David and Semmes showed that, for an $n$-ADR set, the $L^2$ boundedness of all singular integral operators with smooth and odd kernels is equivalent to uniform rectifiability. Later on Tolsa \cite{tolsa2008uniform} improved on this by showing that uniform rectifiability is equivalent to the $L^2$ boundedness of all SIOs with kernels in $\mathcal{K}^n(\R^d)$.
	
	The situation in the non-ADR setting is less clear. A necessary condition for the boundedness of SIOs in $L^2(\mu)$, where $\mu$ is Radon and non-atomic, is the \emph{polynomial growth condition}:
	\begin{equation}\label{eq:growth condition}
		\mu(B(x,r))\le C_1\, r^n,
	\end{equation}
	see \cite[Proposition 1.4 in Part III]{david1991wavelets}. Some sufficient conditions for boundedness of nice SIOs have been shown in \cite{azzam2015characterization} and \cite{girela-sarrion2018}. In \cite{dabrowski2020cones} we showed that BPBE($2$) is another sufficient condition.	
	\begin{theorem}[{\cite[Theorem 1.17]{dabrowski2020cones}}]\label{thm:SIO theorem}
		Let $\mu$ be a Radon measure on $\R^d$ satisfying the polynomial growth condition \eqref{eq:growth condition}.	
		Suppose that $\mu$ has BPBE(2).
		Then, all singular integral operators $T_{\mu}$ with kernels $k\in\mathcal{K}^n(\R^d)$ are bounded in $L^2(\mu)$, with norm depending only on $\alpha, C_1, M_0,\kappa,$ and the constant $C_k$ from \eqref{eq:calderon zygmund constant}.
	\end{theorem}
	
	The result was inspired by \cite[Theorem 10.2]{chang2017analytic} where Chang and Tolsa showed an analogous result with BPBE($2$) replaced by BPBE($1$). It is easy to see that for measures satisfying the polynomial growth condition \eqref{eq:growth condition} we have
	\begin{equation*}
		\mathcal{E}_{\mu,2}(x,V,\alpha,R)\le C_1\,\mathcal{E}_{\mu,1}(x,V,\alpha,R),
	\end{equation*}
	so that BPBE($1$) implies BPBE($2$). In Section \ref{sec:Joyce Morters example} we show that the measure constructed in \cite{joyce2000set} does not satisfy the stronger condition of \cite[Theorem 10.2]{chang2017analytic}, but it trivially satisfies the assumptions of \thmref{thm:SIO theorem}.
	\begin{theorem}
		The measure constructed in \cite{joyce2000set} satisfies BPBE(2), but does not satisfy BPBE(1).
	\end{theorem}
	Hence, \thmref{thm:SIO theorem} really improves on \cite[Theorem 10.2]{chang2017analytic}. This also illustrates the following curious phenomenon: in the case of ADR measures, BPBE($p$) is equivalent to BPBE($q$) for any $p,q\ge 1$, but once we drop the AD regularity assumption, the conditions are no longer equivalent. In the polynomial growth case \eqref{eq:growth condition} we only have
	\begin{equation*}
		\text{BPBE}(p)\quad\Rightarrow\quad \text{BPBE}(q)
	\end{equation*}
	if $p\le q$.
	
	Finally, let us mention that \thmref{thm:SIO theorem} is sharp in the sense that one cannot replace BPBE(2) with BPBE($p$) for any $p>2$, see \cite[Remark 1.20]{dabrowski2020cones}. At the same time, as noted just below \cite[Remark 1.20]{dabrowski2020cones}, BPBE(2) is not a necessary condition for boundedness of nice SIOs.
	
	\subsection*{Acknowledgements} I would like to thank Xavier Tolsa for many helpful conversations, and for his feedback on the first draft of this article.
	
	I received support from the Spanish Ministry of Economy and Competitiveness, through the María de Maeztu Programme for Units of Excellence in R\&D (MDM-2014-0445), and also partial support from the Catalan Agency for Management of University and Research Grants (2017-SGR-0395), and from the Spanish Ministry of Science, Innovation and Universities (MTM-2016-77635-P).
	
	\subsection*{Notation}  We will write $f\lesssim g$ to denote $f\le C g$ for some absolute constant $C>0$. If $C$ depends on parameter $t$, we will write $f\lesssim_t g$. Moreover, $f\approx g$ denotes $f\lesssim g\lesssim f$.
	
	Throughout the article we will be only working with cones in $\R^2$, and so it will be convenient to use the following notation: given $\theta\in [0,\pi)$ let $V_{\theta}\in G(2,1)$ be the line forming angle $\theta$ with the $x$-axis, i.e.
	\begin{equation*}
		V_{\theta} = \{(x,\tan(\theta)x)\, :\, x\in\R\}
	\end{equation*}
	for $\theta\neq\pi/2$, and $V_{\pi/2}$ the vertical axis. We will write
	\begin{equation*}
		K(x,\theta,\alpha):=K(x,V_{\theta},\alpha),
	\end{equation*}
	and similarly $K(x,\theta,\alpha,r):=K(x,V_{\theta},\alpha,r)$.

	Given two lines $V_1,\ V_2$, we denote by $\measuredangle(V_1, V_2)\in [0,\pi/2]$ the angle between $V_1$ and $V_2$.
	\section{Set with BPLG but no BME}\label{sec:BPLG example}
We will show the following.
\begin{prop}\label{prop:counterexample BPLG}
	Fix an aperture parameter $\alpha\in (0,\pi/2)$. There exists a sequence of $1$-ADR sets $E_N=E_N(\alpha)\subset B(0,1)\subset\R^2,\ N\ge 100(1+\log_2(\alpha^{-1})),$ with the following properties:
	\begin{enumerate}
		\item they all contain BPLG in a uniform way, that is, they are $1$-ADR with some absolute constant $C$, and they all satisfy the BPLG condition with $L=1$ and some absolute constant $\kappa>0$.
		\item regardless of the choice of directions $\theta_x\in[0,\pi)$, for all $p\ge 1$ they have big conical energies:
		\begin{multline}\label{eq:big conical energy}
			\int_{E_N}\mathcal{E}_{E_N,p}(x,\theta_x,\alpha,1)\ d\Ho(x)\\			
			= \int_{E_N}\int_0^{1} \bigg(\frac{\Ho(K(x,\theta_x,\alpha,r)\cap E_N)}{r} \bigg)^p\ \frac{dr}{r}d\Ho(x) \gtrsim_{\alpha} N.
		\end{multline}
	\end{enumerate} 
\end{prop}
For the sake of clarity, we will only prove \eqref{eq:big conical energy} for $p=1$ -- the proof for other $p$ is virtually the same. More precisely, to show \eqref{eq:big conical energy} we find a large subset of $E_N$ (with length depending on $\alpha$) such that for any $x$ in the subset and any direction $\theta$ we have $\Ho(K(x,\theta,\alpha,r)\cap E_N)/r\gtrsim 1$ at $\approx N$ distinct dyadic scales $0<r<1$. Thus, $(\Ho(K(x,\theta,\alpha,r)\cap E_N)/r)^p\gtrsim 1$ at the same dyadic scales, which gives \eqref{eq:big conical energy} for arbitrary $p\ge 1$. See the proofs of \lemref{lem:big energy on the bad set} and \lemref{lem:dyadized lots of Gamma in cone}.

Let $\alpha_k\to 0$. Now, a disjoint union of appropriately rescaled sets $E_N(\alpha_{k}),$ with $k,N\to\infty$, would contain BPLG and would not satisfy the BME($p$) condition (Definition \ref{def:BME}) for any $M_0$ and $\alpha>0$. We omit the details.

Without loss of generality we will assume that $\alpha>0$ is smaller than some absolute constant, which is smaller than $\pi/100$, say (note that taking smaller $\alpha$ makes \eqref{eq:big conical energy} more difficult to prove). Let $M=100 \lceil \alpha^{-1}\rceil$, so that $M\approx \alpha^{-1}$. In the lemma below we construct a Lipschitz graph $\Gamma=\Gamma(N,M)$ that can be seen as the first approximation of the set $E_N$. For all directions $\theta$ in $[0,\pi/4]$ the conical energy $\mathcal{E}_{\Gamma,1}(x,V,\alpha,1)$ is bigger than $N$ for all $x$ belonging to a neighbourhood of a large portion of $\Gamma$. Rescaled and rotated copies of $\Gamma$ will be then used as building blocks in the construction of $E_N$. 

Let $\Delta$ be the usual dyadic grid of open intervals on $(-1,1)$, and let $\Delta_k$ denote the dyadic intervals of length $2^{-k}$. 

\begin{lemma}\label{lem:Gamma from EN}
	Let $N\ge 100(1+\log_2(\alpha^{-1}))$ be an integer. There exists a piecewise linear $1$-Lipschitz function $g:[-1,1]\to [-M^{-1},M^{-1}]$, and a collection of disjoint dyadic intervals $\mathcal{I}\subset \Delta$ with the following properties:
	\begin{enumerate}[label={(P\arabic*)}]
		\item $g(-1)=g(1)=0.$ \label{it:endpoints}
		\item For every $I\in\mathcal{I}$ we have $I\subset[-1/2,1/2]$, the function $\restr{g}{I}$ is increasing, and for $t\in I$ we have $g'(t)=1$. \label{it:increasing}
		\item \label{it:intervals}$\#\mathcal{I}=2^{-M}\, 2^{N(M+1)}$ and $\mathcal{I}\subset\Delta_{N(M+1)}$. Hence,
		\begin{equation*}
			\Ho\left(\bigcup_{I\in\mathcal{I}} I\right) = 2^{-M}\approx_{\alpha} 1.
		\end{equation*}
		\item \label{it:big energy}Let $\Gamma=\graph(g)$, $G:[-1,1]\to\Gamma$ be the graph map $G(t)=(t,g(t))$ For any $I\in\mathcal{I}$, any $x\in\R^2$ with $\dist(x,G(I))< 2^{-N(M+1)}$, and all $\theta\in [0,\pi/4]$, we have
		\begin{equation}\label{eq:big energy graph}
			\int_0^1 \frac{\Ho(K(x,\theta,\alpha,r)\cap \Gamma)}{r} \ \frac{dr}{r}\gtrsim N.
		\end{equation}
	\end{enumerate}
\end{lemma}

For an idea of what $\Gamma$ looks like, see the graph at the bottom of Figure \ref{fig:plots of g}. Before we prove \lemref{lem:Gamma from EN}, let us show how it can be used to prove \propref{prop:counterexample BPLG}.

\subsection{Construction of \texorpdfstring{$E_N$}{E\_N}}			
Let $\Gamma=\Gamma(M,N)$ be the $1$-Lipschitz graph from \lemref{lem:Gamma from EN}. The set $E_N$ will consist of one ``big'' Lipschitz graph $\Gamma_0=\Gamma$, and three layers of much smaller Lipschitz graphs stacked on top of the big one. The small graphs will be rescaled and rotated versions of $\Gamma$. Roughly speaking, the big graph ensures big conical energy in directions $[0,\pi/4]$, the first layer of small graphs ensures big conical energy in direction $[\pi/4,\pi/2]$, and so on.

Another way to see $E_N$ is as a union of four bilipschitz curves $\Gamma_0,\dots, \Gamma_3$, and this is how we are going to define it. If $\Gamma_i$ is already defined, $\Gamma_{i+1}$ will be constructed by replacing some of the segments comprising $\Gamma_i$ with rescaled and rotated copies of $\Gamma$.

First, let $\rho:\R^2\to\R^2$ be the counterclockwise rotation by $\pi/4$. Set $L_0=\{(x,0) \, :\, x\in\R\}$ and for $k\ge 1$ set $L_k= \rho^k (L_0)\in G(2,1)$ (here $\rho^k$ denotes $k$ compositions of $\rho$, and the same notation is used for $\delta$ defined below). 

Define also $r_k = 2^{-kN(M+1)-k/2},$ and let $\delta:\R^2\to\R^2$ be the dilation by factor $r_1,$ i.e. $\delta(x)=r_1 x$. Note that $r_k = (r_1)^k$, so that $\delta^k$ is the dilation by factor $r_k$. The constant $r_1$ was chosen in such a way that for an interval $I\in\mathcal{I}\subset\Delta_{N(M+1)}$ we have $\Ho(G(I))=2r_1$ by \ref{it:increasing} (where $G$ is the graph map of $g$).

We will abuse the notation and identify the segment $S_0:=[-1,1]\times \{0\}$ with $[-1,1]\subset\R$. 

Set $\Gamma_0=\Gamma$, and let $\gamma_0=\sigma_0:S_0\to\Gamma_0$ be defined as the natural graph map $\gamma_0(t)=\sigma_0(t)=G(t)=(t,g(t))$.
\begin{lemma} \label{lem:union of curves}
	Let $k\in\{1,2,3\}$. Denote by $\mathcal{I}^k$ the $k$-fold Cartesian product of $\mathcal{I}$, where $\mathcal{I}$ is the family of intervals from \lemref{lem:Gamma from EN}. There exist $\gamma_k:S_0\to \R^2,\ \Gamma_k:=\gamma_k(S_0)$, and for each $I=(I_1,\dots,I_k)\in\mathcal{I}^k$ there exist sets $S_{k,I},\ \Gamma_{k,I}$, and a map $G_{k,I}:\R^2\to\R^2,$ such that:
	\begin{enumerate}[label=\alph*)]
		\item $G_{k,I}:=\tau_I\circ\rho^k\circ\delta^k$ for some translation $\tau_I$, and $S_{k,I}:=G_{k,I}(S_0)$ are segments (in particular, $\Ho(S_{k,I})=2\,r_k$ and $S_{k,I}$ are parallel to $L_k$),
		\item $\Gamma_{k,I}$ are rescaled and rotated copies of $\Gamma$, with $\Gamma_{k,I} := G_{k,I}(\Gamma_0)$ (in particular, since the endpoints of $\Gamma_0$ and $S_0$ coincide, the same is true for $\Gamma_{k,I}$ and $S_{k,I}$),
		\item $\Gamma_k :=\gamma_k(S_0)$ are of the form 
		\begin{equation*}
			\Gamma_{k} = \bigg(\Gamma_{k-1}\setminus\bigcup_{I\in\mathcal{I}^{k}}S_{k,I}\bigg)\cup\bigcup_{I\in\mathcal{I}^{k}}\Gamma_{k,I},
		\end{equation*}
		\item for $k=1,\, J\in\mathcal{I},$ we have $S_{1,J}=\sigma_0(J)\subset\Gamma_0$, and for $k>1$, if $I=(I',J)\in\mathcal{I}^{k-1}\times \mathcal{I}$, then $S_{k,I}=G_{k-1,I'}(S_{1,J})\subset \Gamma_{k-1,I'}\subset \Gamma_{k-1}$,
		\item if $I=(I',J)$, $a_1, a_2$ are the endpoints of $S_{k,I}$, and $b_1, b_2$ are the endpoints of $\Gamma_{k-1,I'}$, then
		\begin{equation*}
			|a_i - b_j|\gtrsim r_{k-1}
		\end{equation*}
		for $i,j\in \{1,2\}$ (i.e. $S_{k,I}$ is ``deep inside'' $\Gamma_{k-1,I'}$),
		\item the maps $\gamma_k$ are of the form $\gamma_k=\sigma_k\circ\dots\circ\sigma_0$, where $\sigma_k:{\Gamma}_{k-1}\to{\Gamma}_k$ is defined as
		\begin{equation*}
			\sigma_k(x) = \begin{cases}
				x,\quad&\text{for}\ x\in\Gamma_{k-1}\setminus\bigcup_{I\in\mathcal{I}^k}S_{k,I}\\
				G_{k,I}(x)\circ\sigma_0\circ G_{k,I}^{-1}(x),\quad &\text{for}\ x\in S_{k,I},\ I\in\mathcal{I}^k.
			\end{cases}
		\end{equation*}				
		In particular, $\sigma_k(S_{k,I})=\Gamma_{k,I}$.
		\item $\lVert \sigma_k - id\rVert_{L^{\infty}(\Gamma_{k-1})}\le 2M^{-1}r_{k},$
	\end{enumerate}
\end{lemma}		
\begin{proof}[Proof of \lemref{lem:union of curves}] 
	
	We will define $\sigma_k$ inductively. 
	
	First, for any $I\in\mathcal{I}$ set $S_{1,I}:=\sigma_0(I)\subset \Gamma_0$. Observe that by \ref{it:increasing} $S_{1,I}$ is a segment parallel to $L_1$. Moreover, since $\Ho(I)=2^{1/2}\,r_1$, we have $\Ho(S_{1,I})=2\,r_1$. It follows that $S_{1,I}=\tau_I\circ\rho\circ\delta(S_0)$ for some translation $\tau_I$. Define $G_{1,I}:\R^2\to\R^2$ as $G_{1,I}= \tau_I\circ\rho\circ\delta$, and $\Gamma_{1,I} = G_{1,I}(\Gamma_0)$.
	
	We define $\sigma_{1}:{\Gamma}_0\to\R^2$ as in f). In other words, $\restr{\sigma_1}{S_{1,I}}$ can be seen as a graph map parametrizing the Lipschitz graph $\Gamma_{1,I}$. It is very easy to see that $S_{1,I},\ \Gamma_{1,I},$ and $\sigma_1$ defined in this way satisfy all the conditions except for e) and g), which we will prove later on.
	
	Now, suppose that $\sigma_{k-1}$, $\gamma_{k-1}$, etc. have already been defined, and that they satisfy a) -- d), f).
	
	For any $I=(I',J)\in\mathcal{I}^{k-1}\times \mathcal{I}$ set $S_{k,I}:=G_{k-1,I'}(S_{1,J})\subset\Gamma_{k-1,I'}$. Since $S_{1,J}$ is parallel to $L_1$ and $G_{k-1,I'} = \tau_{I'}\circ\rho^{k-1}\circ\delta^{k-1}$, $S_{k,I}$ is a segment parallel to $L_{k}$. Moreover, since $\Ho(S_{1,J})=2 \,r_1$, we have $\Ho(S_{k,I})=2\, r_1\, r_{k-1} = 2 r_{k}$. It follows that $S_{k,I}=\tau_I\circ\rho^{k}\circ\delta^{k}(S_0)$ for some translation $\tau_I$.
	
	We define $\sigma_{k}:{\Gamma}_0\to\R^2$ as in f), so that $\restr{\sigma_{k}}{S_{k,I}}$ can be seen as a graph map parametrizing the Lipschitz graph $\Gamma_{k,I}$. It is easy to see that $\sigma_{k}, \Gamma_{k},$ etc. defined this way satisfy a) -- d), f).
	
	\emph{Proof of e).} Let $k=1$. Recall that for all $I\in\mathcal{I}$ we have $I\subset[-1/2,1/2]$ by \ref{it:increasing}. Hence, $S_{1,I} = \sigma_0(I)\subset\sigma_0([-1/2,1/2])\subset \Gamma_0$. If $x\in \sigma_0([-1/2,1/2])$ is arbitrary and if $y\in\Gamma_0$ is one of the endpoints of $\Gamma_0$, we have $|x-y|\gtrsim 1=r_0$. So e) holds for $k=1$. For $k\in\{2,3\}$ the claim follows from the fact that if $I=(I',J)\in\mathcal{I}^{k-1}\times\mathcal{I}$, then $S_{k,I} = G_{k-1,I'}(S_{1,J})$ and $\Gamma_{k-1,I'}= G_{k-1,I'}(\Gamma_0)$.
	
	\emph{Proof of g).} We have $\sigma_k=id$ on $\Gamma_{k-1}\setminus \bigcup_{I\in\mathcal{I}^k} S_{k,I}$, and for $x\in S_{k,I}$
	\begin{multline*}
		|\sigma_k(x)-x|= \big\lvert G_{k,I}\circ\sigma_0\circ G_{k,I}^{-1}(x) - G_{k,I}\circ G_{k,I}^{-1}(x)\big\rvert\\
		= r_k \big\lvert\sigma_0\circ G_{k,I}^{-1}(x) - G_{k,I}^{-1}(x)\big\rvert\le r_k\lVert g\rVert_{\infty} \le 2M^{-1} r_k,
	\end{multline*}
	where we used the fact that $\sigma_0(t)=(t,g(t))$, and that $\lVert g\rVert_{\infty} \le 2M^{-1}$ by \lemref{lem:Gamma from EN}.
\end{proof}
\begin{lemma}\label{lem:sigmak bilip}
	The maps $\gamma_k$ and $\sigma_k$ from \lemref{lem:union of curves} are bilipschitz, with bilipschitz constants independent of $N$ and $\alpha$. 
\end{lemma}
\begin{proof}
	It suffices to show that $\sigma_k$ is bilipschitz with $\lip(\sigma_k)$ and $\lip(\sigma_k^{-1})$ independent of $N,\alpha$, and then the same will be true for $\gamma_k$ by \lemref{lem:union of curves} f).
	
	Suppose that $\sigma_{j}$ are already known to be bilipschitz for $0\le j\le k-1$, with $\lip(\sigma_{j})$ and $\lip(\sigma_{j}^{-1})$ independent of $N,\alpha$ (clearly, the condition holds for $\sigma_0$). Let $x,y\in\Gamma_{k-1}$. Our aim is to show that $|\sigma_k(x)-\sigma_k(y)|\approx |x-y|$.
	
	\emph{Case 1.} $|x-y|>6M^{-1}r_k$.			
	It follows from \lemref{lem:union of curves} g) that
	\begin{equation*}
		|\sigma_k(x)-\sigma_k(y)|\le |x-y|+|\sigma_{k}(x)-x| + |\sigma_{k}(y)-y|\le |x-y| + 4M^{-1}r_k\le 2|x-y|,
	\end{equation*}
	and
	\begin{equation*}
		|\sigma_k(x)-\sigma_k(y)|\ge|x-y|-|\sigma_{k}(x)-x| - |\sigma_{k}(y)-y|\ge |x-y| - 4M^{-1}r_k\ge \frac{1}{3}|x-y|.
	\end{equation*}
	
	\emph{Case 2.} $x,y\in \Gamma_{k-1}\setminus \bigcup_{I\in\mathcal{I}^k} S_{k,I}$.			
	This case is trivial, because $|\sigma_k(x)-\sigma_k(y)|= |x-y|$. 
	
	\emph{Case 3.} $|x-y|\le 6M^{-1}r_k$, and $x,y\in \overline{S_{k,I}}$ for some $I\in\mathcal{I}^k$.  
	Using the fact that $\sigma_0$ is bilipschitz we get
	\begin{multline*}
		|\sigma_k(x)-\sigma_k(y)| = \big\lvert G_{k,I}\circ\sigma_0\circ G_{k,I}^{-1}(x) - G_{k,I}\circ\sigma_0\circ G_{k,I}^{-1}(y)\big\rvert\\
		= r_k  \big\lvert \sigma_0\circ G_{k,I}^{-1}(x) - \sigma_0\circ G_{k,I}^{-1}(y)\big\rvert \approx r_k  \big\lvert G_{k,I}^{-1}(x) - G_{k,I}^{-1}(y)\big\rvert = |x-y|.
	\end{multline*}
	\emph{Case 4.} $|x-y|\le 6M^{-1}r_k$, $x\in S_{k,I}$ for some $I\in\mathcal{I}^k$, and $y\in \Gamma_{k-1}\setminus \overline{S_{k,I}}$. 
	
	We claim that 
	\begin{equation}\label{eq:y in graph}
		y\in\Gamma_{k-1,I'},
	\end{equation}
	where $I=(I',J)\in\mathcal{I}^{k-1}\times\mathcal{I}$ and $\Gamma_{k-1,I'}$ is the Lipschitz graph containing $S_{k,I}$. Indeed, by the induction assumption, the map $\gamma_{k-1}^{-1}:\Gamma_{k-1}\to S_0$ is bilipschitz with $\lip(\gamma_{k-1}),\ \lip(\gamma_{k-1}^{-1})$ independent of $N,\alpha$. Since $\Ho(S_{k,I})=2\, r_k$ and $\Ho(\Gamma_{k-1,I'})\approx r_{k-1}$, we get that
	\begin{equation*}
		\Ho(\gamma_{k-1}^{-1}(S_{k,I}))\approx r_k\quad \text{and}\quad  \Ho(\gamma_{k-1}^{-1}(\Gamma_{k-1,I'}))\approx r_{k-1}.
	\end{equation*}
	Moreover, we have 
	\begin{equation}\label{eq:segments contained}
		\gamma_{k-1}^{-1}(S_{k,I})\subset \gamma_{k-1}^{-1}(\Gamma_{k-1,I'})\subset S_0,
	\end{equation}
	where all three sets are segments. If $a_1, a_2$ and $b_1, b_2$ are the endpoints of $\gamma_{k-1}^{-1}(S_{k,I})$ and $\gamma_{k-1}^{-1}(\Gamma_{k-1,I'}),$ respectively, then it follows from \lemref{lem:union of curves} e) and from the bilipschitz property of $\gamma_{k-1}$ that for $i,j\in\{1,2\}$ we have
	\begin{equation}\label{eq:segments deep inside}
		|a_i-b_j|\gtrsim r_{k-1}.
	\end{equation}
	Recall that $x\in S_{k,I}$ and $|x-y|\lesssim M^{-1}r_k$, so that $\dist(y,S_{k,I})\lesssim M^{-1}r_k$. Hence,
	\begin{equation*}
		\dist(\gamma_{k-1}^{-1}(y),  \gamma_{k-1}^{-1}(S_{k,I}))\lesssim M^{-1}r_k.
	\end{equation*}
	Putting this together with \eqref{eq:segments contained} and \eqref{eq:segments deep inside}, and assuming that $M\ge M_0$ for some absolute constant $M_0>10$, we get that $\gamma_{k-1}^{-1}(y)\in \gamma_{k-1}^{-1}(\Gamma_{k-1,I'})$, which is equivalent to $y\in\Gamma_{k-1,I'}$.
	
	Now, let $z\in \overline{S_{k,I}}$ be an endpoint of $S_{k,I}$ minimizing the distance to $x$. Observe that $x-z\in L_{k}$ and $\sigma_{k}(x)-x\in L_k^{\perp}$, so
	\begin{equation}\label{eq:sigmakx and z decomp}
		|\sigma_k(x) - z|^2 = |\sigma_k(x)-x|^2+|x-z|^2.
	\end{equation}
	Moreover, since $z$ is an endpoint of $S_{k,I}$, the point $G_{k,I}^{-1}(z)$ is an endpoint of $S_0$, and so by \ref{it:endpoints} $g(G_{k,I}^{-1}(z))=0$. Together with the fact that $g$ is $1$-Lipschitz this gives
	\begin{multline}\label{eq:sigmakx and x estimate}
		|\sigma_k(x)-x|= \big\lvert G_{k,I}\circ\sigma_0\circ G_{k,I}^{-1}(x) - G_{k,I}\circ G_{k,I}^{-1}(x)\big\rvert\\
		= r_k \big\lvert\sigma_0\circ G_{k,I}^{-1}(x) - G_{k,I}^{-1}(x)\big\rvert = r_k\big\lvert g(G_{k,I}^{-1}(x))\big\rvert = r_k\big\lvert g(G_{k,I}^{-1}(x)) - g(G_{k,I}^{-1}(z))\big\rvert\\
		\le r_k \big\lvert G_{k,I}^{-1}(x) - G_{k,I}^{-1}(z)\big\rvert = |x-z|.
	\end{multline}
	Furthermore, observe that since $y\in\Gamma_{k-1,I'}\setminus \overline{S_{k,I}},\ z\in \overline{S_{k,I}}$ is an endpoint of $S_{k,I}$, $\Ho(S_{k,I})=2\,r_k$, and $|x-y|\lesssim M^{-1}r_k$, we get that the point $\gamma_{k-1}^{-1}(z)\in S_0$ lies between the points $\gamma_{k-1}^{-1}(x)$ and $\gamma_{k-1}^{-1}(y)$. We already know that $\gamma_{k-1}$ is bilipschitz, and so
	\begin{multline}\label{eq:x z y almost colinear}
		|x-z| + |z-y| \approx |\gamma_{k-1}^{-1}(x)- \gamma_{k-1}^{-1}(z)| + |\gamma_{k-1}^{-1}(z)-\gamma_{k-1}^{-1}(y)|\\
		= |\gamma_{k-1}^{-1}(x)-\gamma_{k-1}^{-1}(y)|\approx |x-y|.
	\end{multline}
	Now, we need to further differentiate between two subcases. 
	
	\emph{Subcase 4a.}  $|x-y|\le 6M^{-1}r_k$, $x\in S_{k,I}$, and $y\in S_{k,Y}$ for some $Y\in\mathcal{I}^k,\ I\neq Y.$
	
	We claim that the point $z$ is a common endpoint of $S_{k,Y}$ and $S_{k,I}$. Indeed, since $y\in\Gamma_{k-1,I'}$ by \eqref{eq:y in graph}, we have $Y=(I', Z)\in\mathcal{I}^{k-1}\times\mathcal{I}$ and $S_{k,Y}\subset\Gamma_{k-1,I'}$. By \lemref{lem:union of curves} d) $S_{k,Y}=G_{k-1,I'}(S_{1,Z}) = G_{k-1,I'}\circ\sigma_0(Z)$, and $S_{k,I}= G_{k-1,I'}\circ\sigma_0(J)$. Recall that $|x-y|\le 6M^{-1}r_k$, which implies $\dist(S_{k,I}, S_{k,Y})\le 6M^{-1}r_k$, and so $\dist(Z,J)\lesssim M^{-1}r_{k-1}^{-1}r_k = M^{-1}r_1$. By \ref{it:intervals} $J$ and $Z$ are dyadic intervals of length $\sqrt{2}\,r_1$, which implies that $\dist(Z,J)=0$. Hence, the point $z$ is a common endpoint of $S_{k,I}$ and $S_{k,Y}$, and the estimates \eqref{eq:sigmakx and z decomp}, \eqref{eq:sigmakx and x estimate} are also valid with $x$ replaced by $y$.
	
	The Lipschitz property of $\sigma_k$ follows easily:
	\begin{equation*}
		|\sigma_k(x) - \sigma_k(y)|\le |\sigma_k(x) - z| + |z - \sigma_k(y)|\overset{\eqref{eq:sigmakx and z decomp},\eqref{eq:sigmakx and x estimate}}{\lesssim}|x-z| + |z-y|\overset{\eqref{eq:x z y almost colinear}}{\approx} |x-y|.
	\end{equation*}
	The converse inequality is a consequence of the fact that $S_{k,I}$ and $S_{k,Y}$ are co-linear, $x-y\in L_k,\ \sigma_k(x)-x\in L_k^{\perp},$ and  $\sigma_k(y)-y\in L_k^{\perp}$:
	\begin{multline*}
		|\sigma_k(x)-\sigma_{k}(y)|^2 = |\sigma_k(x)-x+x-y+y-\sigma_k(y)|^2\\
		= |x-y|^2 + |\sigma_k(x)-x+y-\sigma_k(y)|^2\ge |x-y|^2.
	\end{multline*}
	\emph{Subcase 4b}. $|x-y|\le 6M^{-1}r_k$, $x\in S_{k,I}$ for some $I\in\mathcal{I}^k$, and $y\in \Gamma_{k-1}\setminus \bigcup_{Y\in\mathcal{I}^k} S_{k,Y}$.
	
	In this case we have $\sigma_k(y)=y$. The upper bound follows from previous estimates:
	\begin{equation*}
		|\sigma_k(x) - y|\le |\sigma_k(x) - z| + |z - y|\overset{\eqref{eq:sigmakx and z decomp},\eqref{eq:sigmakx and x estimate}}{\lesssim}|x-z| + |z-y|\overset{\eqref{eq:x z y almost colinear}}{\approx} |x-y|.
	\end{equation*}
	
	Concerning the lower bound, it follows by elementary geometry and properties of our construction that $\pi/4\le\measuredangle(\sigma_k(x),z,y)\le \pi$, see \figref{fig:geometrical}. Thus, using the law of cosines
	\begin{multline*}
		|\sigma_k(x) - y|^2 = |\sigma_k(x) - z|^2 + |z - y|^2 - 2 |\sigma_k(x) - z||z - y|\cos(\measuredangle(\sigma_k(x),z,y))\\
		\ge |\sigma_k(x) - z|^2 + |z - y|^2 - \sqrt{2}|\sigma_k(x) - z||z - y|\\
		\ge \left(1-\frac{\sqrt{2}}{2}\right)(|\sigma_k(x) - z|^2 + |z - y|^2 )\overset{\eqref{eq:sigmakx and z decomp}}{\gtrsim}|x - z|^2 + |z - y|^2\gtrsim |x-y|^2.
	\end{multline*}
	Since this was the last case we had to check, we get that $\sigma_k$ is bilipschitz, as claimed.
\end{proof} 
\begin{figure}[h]
	\centering
	\def\svgwidth{8cm}
	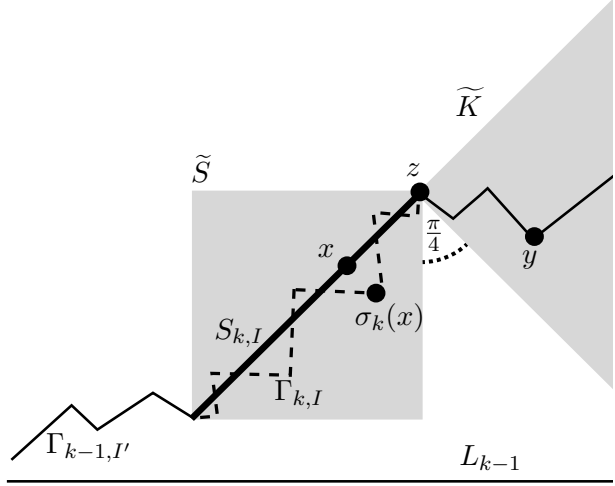
	\caption{Points $x,y,z$ lie on $\Gamma_{k-1,I'}$ (continuous curve above), which is a $1$-Lipschitz graph over the line $L_{k-1}$. $x$ belongs to the segment $S_{k,I}\subset\Gamma_{k-1,I'}$ (thick segment above), and $z$ is an endpoint of $S_{k,I}$. $\sigma_k(x)$ lies on $\Gamma_{k,I}$ (dashed curve above), a $1$-Lipschitz graph over $S_{k,I}$ with the same endpoints as $S_{k,I}$. The $1$-Lipschitz property implies that $\Gamma_{k,I}\subset \widetilde{S},$ where $\widetilde{S}$ is a square having $S_{k,I}$ as diagonal. On the other hand, the $1$-Lipschitz property of $\Gamma_{k-1,I'}$ implies that $\Gamma_{k-1,I'}\subset K_0:=\overline{K(z,L_{k-1}, \pi/4)}$, i.e. it lies in the two-sided version of cone $\widetilde{K}$ above. In particular, $y\in K_0$. However, in Subcase 4b we assume that  $|x-y|\le 6M^{-1}r_k$ and $y\not\in S_{k,I}$, and so $y$ must lie in $\widetilde{K}$, and not in the other one-sided cone comprising $K_0$. Since $L_{k-1}$ and $S_{k,I}$ form an angle $\pi/4$, the observations above imply $\pi/4\le\measuredangle(\sigma_k(x),z,y)\le \pi$ (see the dotted angle).}
	\label{fig:geometrical}
\end{figure}

Finally, we set
\begin{equation*}
	E_N = \Gamma_0\cup\Gamma_1\cup\Gamma_2\cup\Gamma_3.
\end{equation*}
Note that due to \lemref{lem:union of curves} c)
\begin{equation}\label{eq:graph decomp}
	E_N = \Gamma_0\cup\bigcup_{I\in\mathcal{I}}\Gamma_{1,I}\cup\bigcup_{I\in\mathcal{I}^2}\Gamma_{2,I}\cup\bigcup_{I\in\mathcal{I}^3}\Gamma_{3,I}.
\end{equation}
That is, $E_N$ is a union of a single big Lipschitz graph, and three layers of smaller Lipschitz graphs.

\subsection{\texorpdfstring{$E_N$ has BPLG}{E\_N has BPLG}}		
In this section we show that $E_N$ has big pieces of Lipschitz graphs, with constants independent of $N$. 

Observe that $E_N$ is AD-regular because it is a union of four bilipschitz curves. The ADR constants do not depend on $N$ due to \lemref{lem:sigmak bilip}.

\begin{lemma}\label{lem:BPLG of EN}
	For any $x\in E_N$ and any $0<r<\diam(E_N)$ we can find a Lipschitz graph $\Sigma$ (depending on $x$ and $r$) such that
	\begin{equation}\label{eq:BPLG}
		\Ho(E_N\cap B(x,r)\cap\Sigma) \gtrsim r,
	\end{equation}
	with the implicit constant independent of $N,\alpha$.
\end{lemma} 

First, we prove an auxiliary estimate. Given integers $i,l\in\{0,1,2,3\}$ define $\gamma_{l,i}:\Gamma_l\to\Gamma_i$ as $ \gamma_{l,i}=\gamma_i\circ\gamma_l^{-1}.$
\begin{lemma}
	Let $i,l\in\{0,1,2,3\}$ and $k=\min(i,l)$. Then
	\begin{equation}\label{eq:x close to y}
		\lVert \gamma_{l,i} - id\rVert_{L^{\infty}(\Gamma_l)} \le 6M^{-1} r_{k+1}.
	\end{equation}
\end{lemma}
\begin{proof}
	If $i=l$ the result is clear because $\gamma_{l,i}= id$. Assume $l> i$.
	Applying $(l-i)$-many times \lemref{lem:union of curves} g) we get that
	\begin{multline*}
		|x-\gamma_{l,i}(x)|\le \sum_{j=i+1}^{l} |\gamma_{l,j-1}(x) - \gamma_{l,j}(x)| = \sum_{j=i+1}^{l} |\gamma_{l,j-1}(x) - \sigma_j(\gamma_{l,j-1}(x))|\\
		\le \sum_{j=i+1}^{l} 2M^{-1}r_j\le 2(j-k)M^{-1} r_{i+1}\le 6M^{-1} r_{i+1}.
	\end{multline*}
	On the other hand, if $l < i$, then applying the estimate above to $y=\gamma_{l,i}(x)$ we get 
	\begin{equation*}
		|x-\gamma_{l,i}(x)| = |\gamma_{i,l}(y) - y|\le  6M^{-1} r_{l+1}.
	\end{equation*}
\end{proof}

\begin{proof}[Proof of \lemref{lem:BPLG of EN}]
	Let $x\in E_N$ and $0<r<\diam(E_N).$ By \eqref{eq:graph decomp} there exist $j\in\{0,1,2,3\}$ and $I\in\mathcal{I}^j$ such that $x\in\Gamma_{j,I}$. 
	
	Suppose $r<r_j$. Since $\Gamma_{j,I}$ is a Lipschitz graph satisfying $\Ho(\Gamma_{j,I})\ge r_j>r$, we have
	\begin{equation*}
		\Ho(E_N\cap B(x,r)\cap\Gamma_{j,I})= \Ho(B(x,r)\cap\Gamma_{j,I})\gtrsim r.
	\end{equation*}
	That is, we may choose $\Sigma=\Gamma_{j,I}$.
	
	Now assume $r_j\le r< r_0=1$. Let $k\in\{0,1,2\}$ be such that $r_{k+1}\le r < r_{k}$ (of course, $k+1\le j$).  Let $y =\gamma_{j,k}(x)$.
	Observe that, by \lemref{lem:union of curves} c), since $y\in\Gamma_k$, there exists some $k'\in\{0,\dots, k\}$ such that $y\in\Gamma_{k',I'}$ for some $I'\in\mathcal{I}^{k'}$. Since $k'\le k$, we have $\Ho(\Gamma_{k',I'})\approx r_{k'}\ge r_k> r$. Moreover, assuming $M\ge 12$, \eqref{eq:x close to y} gives
	\begin{equation*}
		\dist(x,\Gamma_{k',I'})\le |x-y| = |x-\gamma_{j,k}(x)|\le \frac{r_{k+1}}{2}\le \frac{r}{2},
	\end{equation*}
	and so
	\begin{equation*}
		\Ho(E_N\cap B(x,r)\cap\Gamma_{k',I'}) = \Ho(B(x,r)\cap\Gamma_{k',I'})\gtrsim r.
	\end{equation*}
	Hence, we may choose $\Sigma=\Gamma_{k',I'}$.
	
	Finally, for $1<r<\diam(E_N)\approx 1$, the condition \eqref{eq:BPLG} is satisfied with $\Sigma=\Gamma_0$.
\end{proof}

\subsection{\texorpdfstring{$E_N$}{E\_N} has big conical energy}
In this section we show that $E_N$ satisfies \eqref{eq:big conical energy}.

We introduce additional notation. Analogously to the definition of $S_{k,I}$ for $k\in\{0,1,2,3\}$, for $I=(I',J)\in\mathcal{I}^3\times\mathcal{I}$ we define $S_{4,I}=G_{3,I'}(S_{1,J})$.		

If $I\in\mathcal{I}^{k+j}$ is of the form $I=(I',I'')\in \mathcal{I}^k\times\mathcal{I}^j$, we will write
\begin{equation*}
	S_{k,I}:=S_{k,I'},\qquad \Gamma_{k,I}:=\Gamma_{k,I'},\qquad G_{k,I}:=G_{k,I'}.
\end{equation*}

\begin{lemma}\label{lem:big energy on the bad set}
	Let $I=(I_1,I_2,I_3,I_4)\in\mathcal{I}^4$, and let $x\in S_{4,I}\subset\Gamma_{3,I}\subset E_N$. Then, for any $\theta\in[0,\pi)$ we have 
	\begin{equation}\label{eq:big energy}
		\int_0^1 \frac{\Ho(K(x,\theta,\alpha,r)\cap E_N)}{r} \ \frac{dr}{r}\gtrsim N.
	\end{equation}
\end{lemma}
\begin{proof}
	Let $I\in\mathcal{I}^4$, $x\in S_{4,I}$ and $\theta\in[0,\pi)$ be as above. Recall that $L_0=\{(x,0) \ :\ x\in\R\}$, $\rho$ is the counterclockwise rotation by $\pi/4$, and $L_k= \rho^k (L_0)\in G(2,1).$ 
	Observe that there exists some $k\in\{0,1,2,3\}$ such that $\theta-k\pi/4\in[0,\pi/4)$. Fix such $k$. 
	We are going to use \ref{it:big energy} with respect to $\Gamma_{k,I}$ to arrive at \eqref{eq:big energy}.
	
	Recall that $S_{k+1,I}=G_{k,I}(S_{1,I_{k+1}}),$ where $G_{k,I}= \tau\circ\rho^k\circ\delta^k$ for some translation $\tau$. Recall also that $G_{k,I}(\Gamma) = \Gamma_{k,I}$. Let $x' = G_{k,I}^{-1}(x)$, and $\theta'=\theta-k\pi/4\in[0,\pi/4)$. Note that $V_{\theta'}=\rho^{-k}(V_{\theta})=G_{k,I}^{-1}(V_{\theta})$. Using the fact that $G_{k,I}$ is a similarity with stretching factor $r_k$, we get
	\begin{multline}\label{eq:estimating conical energu again}
		\int_0^1 \frac{\Ho(K(x,\theta,\alpha,r)\cap E_N)}{r} \ \frac{dr}{r}\ge \int_0^1 \frac{\Ho(K(x,\theta,\alpha,r)\cap \Gamma_{k,I})}{r} \ \frac{dr}{r}\\
		=\int_0^1 r_k\, \frac{\Ho(K(x',\theta',\alpha,r_k^{-1}\,r)\cap \Gamma)}{r} \ \frac{dr}{r} = \int_0^{r_k^{-1}} \frac{\Ho(K(x',\theta',\alpha,s)\cap \Gamma)}{s} \ \frac{ds}{s}. 
	\end{multline}
	Recall that $k$ was chosen in such a way that $\theta'\in[0,\pi/4)$. In order to use \ref{it:big energy}, it only remains to show that $\dist(x',G(I'))\le 2^{-N(M+1)}$ for some $I'\in \mathcal{I}$.
	
	Observe that $\gamma_{3,k}(S_{4,I})\subset S_{k+1,I}$. We know from \eqref{eq:x close to y} that if $M\ge 6$, then
	\begin{equation}\label{eq:dist x to Sk+1}
		\dist(x,S_{k+1,I})\le \dist(x,\gamma_{3,k}(S_{4,I}))\le |x-\gamma_{3,k}(x)|\le  r_{k+1}.
	\end{equation}
	Thus,
	\begin{multline*}
		\dist(x', S_{1,I_{k+1}}) = \dist(G_{k,I}^{-1}(x),G_{k,I}^{-1}(S_{k+1,I})) = r_k^{-1}\dist(x,S_{k+1,I})\\
		\le r_k^{-1}\, r_{k+1}  = r_1=2^{-N(M+1)-1/2}\le 2^{-N(M+1)}.
	\end{multline*}
	$S_{1,I_{k+1}}$ was defined as $\sigma_0(I_{k+1})=G(I_{k+1})$, and so it follows from \ref{it:big energy} that the last term in \eqref{eq:estimating conical energu again} is greater than $CN$ for some absolute constant $C$. Thus, \eqref{eq:big energy} holds.
\end{proof}
Now we can finish the proof of \propref{prop:counterexample BPLG}.		
Observe that
\begin{multline*}
	\Ho\bigg(\bigcup_{I\in\mathcal{I}^4} S_{4,I} \bigg) = \sum_{I'\in\mathcal{I}^3,\, J\in\mathcal{I}}\Ho(G_{3,I'}(S_{1,J}))\\
	= (\#\mathcal{I})^3\, r_3\sum_{J\in\mathcal{I}} \Ho(S_{1,J}) \ge (\#\mathcal{I})^3\, r_3\sum_{J\in\mathcal{I}}\Ho(J)\\
	\overset{\ref{it:intervals}}{=}2^{-3M}\, 2^{3N(M+1)}\, 2^{-3N(M+1)-3/2}\, 2^{-M+1} = 2^{-4M-1/2}\approx_{\alpha} 1,
\end{multline*}
where we also used that $M$ is a constant depending only on $\alpha$. Together with \lemref{lem:big energy on the bad set}, this shows that the set $E_N$ has the desired property \eqref{eq:big conical energy}, i.e.
\begin{multline*}
	\int_{E_N}\int_0^{1} \frac{\Ho(K(x,\theta_x,\alpha,r)\cap E_N)}{r} \ \frac{dr}{r}d\Ho(x)\\
	\ge \sum_{I\in\mathcal{I}4} \int_{S_{4,I}}\int_0^{1} \frac{\Ho(K(x,\theta_x,\alpha,r)\cap E_N)}{r} \ \frac{dr}{r}d\Ho(x) \gtrsim_{\alpha} N.
\end{multline*}
Thus, the proof of \propref{prop:counterexample BPLG} is complete. All that remains to prove is \lemref{lem:Gamma from EN}. We do that in the following two subsections.

\subsection{Construction of \texorpdfstring{$g$}{g}}

In this subsection we construct a function $g$ and a family of dyadic intervals $\mathcal{I}$ that satisfy \ref{it:endpoints}, \ref{it:increasing}, and \ref{it:intervals}.

First, we define a family of auxiliary functions.
For $j=1,\dots, M$ we define $f_j:[-1,1]\to [-M^{-1}2^{-jN},M^{-1}2^{-jN}]$ as
\begin{equation*}
	f_j(t) = \frac{h(2^{jN}t)}{M2^{jN}},
\end{equation*}
where $h(t):\R\to [-1,1]$ is the $1$-Lipschitz triangle wave:
\begin{equation*}
	h(t) = |t \bmod 4 -2| - 1.
\end{equation*}
In the above $t \bmod 4$ denotes the unique number $s\in [0,4)$ such that $t=4k+s$ for some $k\in\Z$. 

Note that for all $j$ we have $\lip(f_j)=M^{-1}$. For $j=1,\dots, M$ we define also $g_j:[-1,1]\to [-M^{-1}2^{-N+1},M^{-1}2^{-N+1}]$ as
\begin{equation*}
	g_j(t) = \sum_{i=1}^{j} f_{i}(t),
\end{equation*}
and we set $\Gamma_j = \graph(g_j)\subset B(0,1)\subset\R^2,\ g=g_M,\ \Gamma=\Gamma_M$.	See \figref{fig:plots of g}.
Observe that $g$ is $1$-Lipschitz. 
\begin{figure}[h]
	\includegraphics[height=8cm]{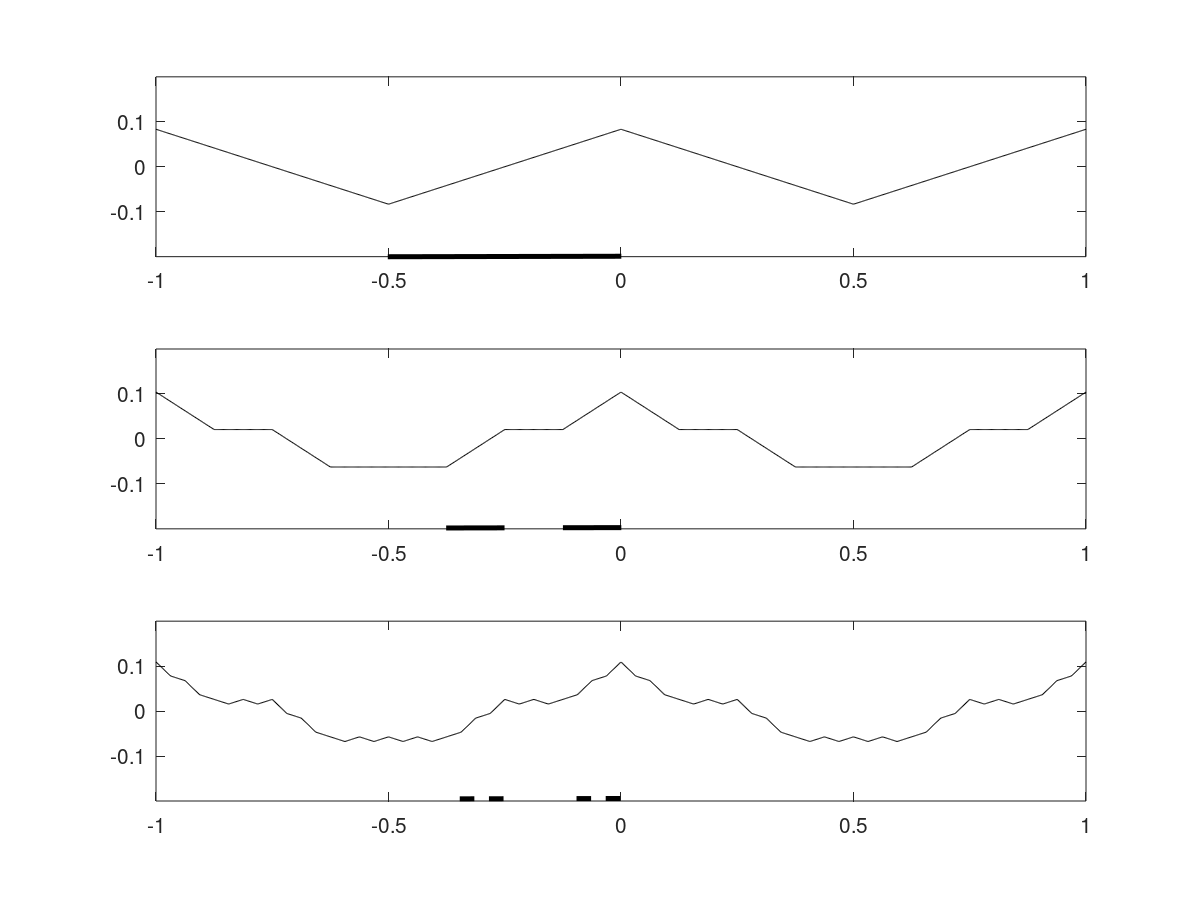}
	\caption{Top to bottom: graphs of $g_1,\ g_2,$ and $g_3=g$ when $N=2$ and $M=3$. The thick segments denote intervals in $\mathcal{G}_1,\ \mathcal{G}_2,$ and $\mathcal{G}_3,$ respectively.} 
	\label{fig:plots of g}
\end{figure}
\begin{proof}[Proof of \ref{it:endpoints}]
	We want to show that $g(1)=g(-1)=0$. Since $h$ is an even function, the functions $f_j$ and $g_j$ are also even. Hence, $g(1)=g(-1)$. Note also that if we have some function $\tilde{g}$ satisfying properties \ref{it:increasing} and \ref{it:big energy}, then for any constant $C\in\R$ the function $\tilde{g}+C$ will also satisfy \ref{it:increasing} and \ref{it:big energy}. In other words, these properties are invariant under adding constants. It follows that we can work with the function $g$ as defined above, prove \ref{it:increasing} and \ref{it:big energy}, and at the end replace $g$ by $g-g(1)$. So the property \ref{it:endpoints} is not an issue.
\end{proof}

We proceed to define the family $\mathcal{I}\subset\Delta_{(M+1)N}$.

Recall that $\Delta_k$ denotes the open dyadic intervals of length $2^{-k}$. Observe that for any $j$ the functions $f_j$ and $g_j$ are linear on each interval from $\Delta_{jN}$, and we have $f_j'=M^{-1}$ on every second interval, and $f_j'=-M^{-1}$ on the rest. 

Set $\mathcal{G}_j\subset\Delta_{jN}$ to be the family of dyadic intervals $I$ contained in $[-1/2,1/2]$ such that for all $1\le i\le j$ we have $f'_i = M^{-1}$ on $I$. It is easy to see that each $\mathcal{G}_j$ consists of $2^{jN -j}$ disjoint intervals of length $2^{-jN}$, see \figref{fig:plots of g}. We define also $\mathcal{I}\subset \Delta_{(M+1)N}$ as the family of dyadic intervals of length $2^{-(M+1)N}$ contained in $\bigcup_{I\in\mathcal{G}_M} I$. 

\begin{proof}[Proof of \ref{it:intervals}]
	By the definition above we have
	\begin{equation}\label{eq:number of intervals in I}
		\#\mathcal{I}=2^N\cdot\#\mathcal{G}_M = 2^{(M+1)N-M},
	\end{equation}
	so the property \ref{it:intervals} holds.
\end{proof}

\begin{proof}[Proof of \ref{it:increasing}]
	We have defined $\mathcal{G}_j$ in such a way that if $t\in I\in\mathcal{G}_j $ then $g_j'(t)=jM^{-1}$. It follows that if $t\in I\in\mathcal{I}$, then $t\in J$ for some $J\in\mathcal{G}_M$, and so $g'=1$. Thus, \ref{it:increasing} holds.		
\end{proof}

\subsection{\texorpdfstring{$\Gamma$}{Gamma} has big conical energy}
This subsection is dedicated to proving \ref{it:big energy}. We recall the statement for reader's convenience: 
\begin{enumerate}
	\item[(P4)] Let $\Gamma=\graph(g)$, $G:[-1,1]\to\Gamma$ be the graph map $G(t)=(t,g(t))$. For any $I\in\mathcal{I}$, any $x\in\R^2$ with $\dist(x,G(I))< 2^{-N(M+1)}$, and all $\theta\in[0,\pi/4]$, we have
	\begin{equation}\label{eq:big conical energy of Gamma}
		\int_0^1 \frac{\Ho(K(x,\theta,\alpha,r)\cap \Gamma)}{r} \ \frac{dr}{r}\gtrsim N.
	\end{equation}
\end{enumerate}

Fix $x,\ I,$ and $\theta$ as above. We will show \eqref{eq:big conical energy of Gamma}.

Since $\dist(x,G(I))< 2^{-N(M+1)}$, there exists $t_0\in {I}$ such that $|x-G(t_0)|\le 2^{-N(M+1)}$. Fix such $t_0$.	

For every $j=1,\dots, M$ define $G_j(t) = (t,g_j(t))$. For every $t\in \bigcup_{I\in\Delta_{jN}} I$ set $L_j(t)\subset\R^2$ to be the line tangent to $\Gamma_j$ at $G_j(t)$. We define also $I_j(t)$ as the unique interval from $\Delta_{jN}$ containing $t$. Note that, since $g_j$ is linear on intervals from $\Delta_{jN}$, we have $L_j(t)=L_j(t')$ whenever $t'\in I_j(t)$. 
Denote by $L_0$ the $x$-axis.

Observe that if $I_M(t)\in\mathcal{G}_M$, then for each $1\le j\le M$ we have $g_j'(t)=jM^{-1}$. Thus,
\begin{equation}\label{eq:angle Lj L0}
	\measuredangle(L_j(t),L_0)=\arctan(jM^{-1}),\quad\text{and}\quad \measuredangle(L_j(t),V_{\pi/4})\le \pi/8.
\end{equation}
Set $L_j=L_j(t)-(t,g_j(t))$. Note that $(0,0)\in L_j$, and that the definition of $L_j$ does not depend on $t$, as long as $I_M(t)\in\mathcal{G}_M$. 
Since $\theta\in[0,\pi/4]$, it follows from \eqref{eq:angle Lj L0} that there exists some $1\le j\le M$ such that
\begin{multline}
	\measuredangle(V_{\theta},L_j)\le \max_{1\le i \le M} \big(\arctan(iM^{-1}) - \arctan((i-1)M^{-1})\big)\\
	 = \arctan(M^{-1})\le M^{-1}.
\end{multline}	
Fix such $j$. Recall that $M=100 \lceil \alpha^{-1}\rceil$, and so
\begin{equation}
	\measuredangle(V_{\theta},L_j)\le M^{-1}\le \frac{\alpha}{10}.
\end{equation}
Hence, for any $r>0$
\begin{equation}\label{eq:cone around V contains cone around Lj}
	K(x,\theta,\alpha,r)\supset K(x,L_j,\alpha/2,r).
\end{equation}
\begin{lemma}\label{lem:Gj approximates G well}
	For $t\in [-1,1]$ we have $|G(t) - G_j(t)|= |g(t) - g_j(t)|\le2M^{-1}\, 2^{-N(j+1)}$.
\end{lemma}
\begin{proof}
	The estimate follows immediately from the definition of $g$ and $g_j$:
	\begin{equation*}
		|g(t) - g_j(t)| = \left\lvert \sum_{i=j+1}^M f_i(t)\right\rvert \le  \sum_{i=j+1}^M |f_i(t)| \le \sum_{i=j+1}^{\infty} \frac{1}{M}2^{-iN}\le 2M^{-1}\, 2^{-N(j+1)}.
	\end{equation*}
\end{proof}
Recall that $t_0\in I\in\mathcal{I}$ was such that $|x-G(t_0)|\le 2^{-N(M+1)}$. Set $x' = G_j(t_0)$. Then, by the lemma above, we have
\begin{equation}\label{eq:x and x' distance}
	|x-x'|\le |x-G(t_0)| + |G(t_0)-G_j(t_0)|\le 2^{-N(M+1)}+2^{-N(j+1)}\le 2^{-N(j+1)+1}.
\end{equation}
Let $I'\in\mathcal{G}_j$ be the unique dyadic interval in $\Delta_{jN}$ containing $I$. That is, $I' = I_j(t_0)$. 

Let $K(x,V,\alpha,r,R)$ denote the twice truncated cone $K(x,V,\alpha,R)\setminus B(x,r)$. In the lemma below we show that for all the scales between $2^{-N(j+1)}$ and $2^{-Nj}$, $G(I')$ has large intersection with the the twice truncated cone centered at $x'$ with direction $L_j$ corresponding to that scale.

\begin{lemma}
	For $t\in I'$ such that $|G(t)-x'|\ge 2^{-N(j+1)}$ we have $G(t)\in K(x',L_j,\alpha/8).$ Moreover, for integers $k$ satisfying $Nj\le k\le N(j+1) - 1$ we have
	\begin{equation}\label{eq:lower bound on Gamma in cone}
		\Ho (G(I')\cap K(x',L_j,\alpha/8,2^{-k-1},2^{-k+2}))\gtrsim 2^{-k}.
	\end{equation}
\end{lemma}
\begin{proof}
	Let $t\in I'$ satisfy $|G(t)-x'|\ge 2^{-N(j+1)}$. Recall that, since $I'\in\mathcal{G}_j$, the set $G_j(I')$ is a segment parallel to $L_j$. We also know that $x'=G_j(t_0)\in G_j(I')$, and so by \lemref{lem:Gj approximates G well} 
	\begin{equation*}
		\dist(G(t),L_j+x')\le |G(t)-G_j(t)|\le 2M^{-1}\, 2^{-N(j+1)}\le \sin({\alpha}/{8})|G(t)-x'|,
	\end{equation*}
	where we also used that $M=100\lceil \alpha^{-1}\rceil$ and we assume $\alpha$ to be so small that ${\alpha}/8\le 2\sin({\alpha}/{8})$. Thus, $G(t)\in K(x',L_j,\alpha/8)$.
	
	Now, let $k$ be an integer such that $Nj\le k\le N(j+1) - 1$. Let $t\in I'$ be such that $2^{-k}<|t-t_0|$, so that
	\begin{multline*}
		|G(t)-x'|\ge |G_j(t)-G_j(t_0)|-|G(t)-G_j(t)|\ge |t-t_0| - 2M^{-1}\,2^{-N(j+1)}\\
		\ge 2^{-k} - 2M^{-1}\,2^{-N(j+1)}\ge 2^{-N(j+1)}.
	\end{multline*}
	Hence, by our previous result, $G(t)\in K(x',L_j,\alpha/8)$. At the same time, the calculation above shows that $|G(t)-x'|\ge 2^{-k-1}$. Similarly,
	\begin{equation*}
		|G(t)-x'| \le |G_j(t)-G_j(t_0)|+|G(t)-G_j(t)|\le \sqrt{2}|t-t_0| + 2M^{-1}\,2^{-N(j+1)}.
	\end{equation*}
	Hence, for $t\in I'$ such that $2^{-k}\le|t-t_0|\le 2^{-k+1}$ we have 
	\begin{equation*}
		2^{-k-1}\le |G(t)-x'| \le 2^{-k+2}.
	\end{equation*}
	That is, for $t\in I'$ with $2^{-k}\le|t-t_0|\le 2^{-k+1}$ we have
	\begin{equation*}
		G(t)\in K(x',L_j,\alpha/8,2^{-k-1},2^{-k+2}).
	\end{equation*}
	Since $G$ is bilipschitz, \eqref{eq:lower bound on Gamma in cone} follows.
\end{proof}

Later on we will need the following simple lemma about the inclusions of twice truncated cones.
\begin{lemma}\label{lem:truncated cone included in another}
	Let $x_1, x_2\in\R^2$, $L\in G(2,1),\ r>0$ and $\alpha_0\in (0,\pi/50)$. Suppose that $|x_1-x_2|\le \sin(\alpha_0)\,r$. Then
	\begin{equation*}
		K(x_1,L,\alpha_0,\sin(\alpha_0)^{-1}|x_1-x_2|, r) \subset K(x_2,L,8 \alpha_0,2 r).
	\end{equation*} 
\end{lemma}
\begin{proof}
	Let $y\in K(x_1,L,\alpha_0,\sin(\alpha_0)^{-1}|x_1-x_2|, r)$, so that $\sin(\alpha_0)^{-1}|x_1-x_2|<|y-x_1|\le r$ and $\dist(y,L+x_1)\le \sin({\alpha_0})|y-x_1|$. It is clear that for any $p\in L+x_1$ we have $\dist(p,L+x_2) = |x_1-x_2|$, and so
	\begin{multline*}
		\dist(y,L+x_2)\le \dist(y,L+x_1) + |x_1-x_2|\le \sin(\alpha_0)|y-x_1| + \sin(\alpha_0) |y-x_1|\\
		\le 2\sin(\alpha_0) |y-x_2| + 2\sin(\alpha_0)|x_1-x_2|.
	\end{multline*}
	It is easy to check that for $\alpha_0\in (0,\pi/50)$ we have $4\sin(\alpha_0)\le\sin(8\alpha_0)$, and so	
	\begin{equation*}
		\dist(y,L+x_2)\le \frac{\sin(8\alpha_0)}{2}(|y-x_2| + |x_1-x_2|).
	\end{equation*}
At the same time, we have
	\begin{equation*}
		|y-x_2|\ge |y-x_1|-|x_1-x_2|\ge (\sin(\alpha_0)^{-1}-1)|x_1-x_2|\ge |x_1-x_2|.
	\end{equation*}
	Putting the two estimates together gives $y\in K(x_2,L,8 \alpha_0)$. To see that $y\in B(x_2, 2r)$, note that $|y-x_2|\le |y-x_1|+|x_1-x_2|\le 2r$.
\end{proof}

Recall that in \eqref{eq:lower bound on Gamma in cone} we showed a lower bound on the length of intersection of $G(I')$ with a cone centered at $x'$. However, to prove \eqref{eq:big conical energy of Gamma} we need information about the intersections with cones centered at $x$. We use \eqref{eq:lower bound on Gamma in cone} and \lemref{lem:truncated cone included in another} to get the following.

\begin{lemma}\label{lem:dyadized lots of Gamma in cone}
	Let $k$ be an integer such that $\alpha^{-1}\,2^{-N(j+1)+9}< 2^{-k}\le 2^{-Nj-3}$. Then, we have
	\begin{equation}\label{eq:dyadized lots of Gamma in cone}
		\Ho (G(I')\cap K(x,L_j,\alpha/2,2^{-k}))\gtrsim 2^{-k},
	\end{equation}
\end{lemma}
\begin{proof}		
	First, recall that $x'=G_j(t_0)$ and $|x-x'|\le 2^{-N(j+1)+1}$ by \eqref{eq:x and x' distance}. Using our assumptions on $k$, and that we assume $\alpha$ to be so small that $\sin(\alpha/8)\ge \alpha/16$, we get
	\begin{equation}\label{eq:weird estimate of x x'}
		\sin(\alpha/8)^{-1}|x-x'|\le \alpha^{-1}\,2^{-N(j+1)+5}\le2^{-k-4}< 2^{-k-1}.
	\end{equation}
	Hence, we may apply \lemref{lem:truncated cone included in another} with $x_1=x',\ x_2=x,\ L=L_j,\ \alpha_0 = \alpha/8,\ r=2^{-k-1},$ to get
	\begin{equation*}
		K(x',L_j,\alpha/8,\sin(\alpha/8)^{-1}|x-x'|,2^{-k-1})\subset K(x,L_j,\alpha/2,2^{-k}).
	\end{equation*}
	Since $\sin(\alpha/8)^{-1}|x-x'|\le 2^{-k-4}$ by \eqref{eq:weird estimate of x x'}, it follows from the above that
	\begin{equation}\label{eq:one cone included in another}
		K(x',L_j,\alpha/8,2^{-k-4},2^{-k-1})\subset K(x,L_j,\alpha/2,2^{-k}).
	\end{equation}
	Note that we have $Nj\le k-3\le N(j+1)-1$ due to our assumptions on $k$. Thus, we may use
	\eqref{eq:lower bound on Gamma in cone} to get 
	\begin{equation*}
		\Ho (G(I')\cap K(x,L_j,\alpha/8,2^{-k-4},2^{-k-1}))\gtrsim 2^{-k}.
	\end{equation*}
	Together with \eqref{eq:one cone included in another}, this concludes the proof.
\end{proof} 

We are ready to finish the proof of \lemref{lem:Gamma from EN}.
\begin{proof}[Proof of \ref{it:big energy}]
	We want to show that
	\begin{equation}\label{eq:big energy of Gamma again}
		\int_0^1 \frac{\Ho(K(x,\theta,\alpha,r)\cap \Gamma)}{r} \ \frac{dr}{r}\gtrsim N.
	\end{equation}
	We use \eqref{eq:cone around V contains cone around Lj} to write
	\begin{multline}\label{eq:integrals with cones}
		\int_0^1 \frac{\Ho(K(x,\theta,\alpha,r)\cap \Gamma)}{r} \ \frac{dr}{r}\ge \int_0^1 \frac{\Ho(K(x,L_j,\alpha/2,r)\cap \Gamma)}{r} \ \frac{dr}{r}\\
		\ge \int_{\alpha^{-1}\,2^{-N(j+1)+10}}^{2^{-Nj-3}} \frac{\Ho(K(x,L_j,\alpha/2,r)\cap \Gamma)}{r} \ \frac{dr}{r}.
	\end{multline}
	Note that $\alpha^{-1}\,2^{-N(j+1)+10} < 2^{-Nj-3}$ due to the assumption $N\ge 100(1+\log_2(\alpha^{-1}))$. Now let $\alpha^{-1}\,2^{-N(j+1)+10} \le r< 2^{-Nj-3}$, and let $k$ be the unique integer such that $2^{-k}\le r<2^{-k+1}$. Then, $k$ satisfies the assumptions of \lemref{lem:dyadized lots of Gamma in cone}, and we get
	\begin{equation*}
		\Ho (K(x,L_j,\alpha/2,r)\cap\Gamma)\ge \Ho (K(x,L_j,\alpha/2,2^{-k})\cap\Gamma)\gtrsim 2^{-k}\approx r.
	\end{equation*}
	It follows from \eqref{eq:integrals with cones} and the above that 
	\begin{multline*}
		\int_0^1 \frac{\Ho(K(x,V,\alpha,r)\cap \Gamma)}{r} \ \frac{dr}{r}\gtrsim \int_{\alpha^{-1}\,2^{-N(j+1)+10}}^{2^{-Nj-3}} 1 \ \frac{dr}{r}\\ 
		= \log(2)\,( N(j+1) - 10 - \log_2(\alpha^{-1}) - Nj - 3) \\
		= \log(2)(N - \log_2(\alpha^{-1}) - 13)\ge \frac{N}{100},
	\end{multline*}
	where we used the assumption $N\ge 100(1+\log_2(\alpha^{-1}))$ in the last inequality. Thus, the proof of \eqref{eq:big energy of Gamma again} is finished.
\end{proof}

\section{Example of Joyce and M\"{o}rters}\label{sec:Joyce Morters example}
In this section we will show that the measure $\mu$ constructed in \cite{joyce2000set} satisfies BPBE(2), but does not satisfy BPBE(1). Hence, \thmref{thm:SIO theorem} is a true improvement on its $\mathcal{E}_{\mu,1}$ analogue \cite[Theorem 10.2]{chang2017analytic}.

\subsection{Construction of \texorpdfstring{$\mu$}{mu}}
For reader's convenience, we sketch out the construction of Joyce and M\"{o}rters below.

Let $M\ge 3$ be a large constant, and $1/2<\beta_k<1$ be a sequence of numbers converging to $1$. For $k\ge 1$ we define $m_k = M k$, $m(k) = m_1\dots m_k = M^k\, k!$, and
\begin{equation*}
	\sigma_k = \left(\frac{k+1}{k}\right)^{\beta_k}.
\end{equation*}
We set also $\alpha_j = 2^{-n}\,\pi$ for all $2^n\le j< 2^{n+1},\ n\ge 0$.

We proceed to define a compact set $E\subset\R^2$ on which the measure $\mu$ will be supported. First, let $E_0$ be a closed ball of diameter $1$. We place $m_1$ closed balls of diameter $2r_1:=\sigma_1/m_1$ inside $E_0$. We do it in such a way, that 
\begin{itemize}
	\item their centers lie on the diameter of $E_0$ forming angle $\alpha_1$ with the $x$ axis,
	\item the boundaries of the first and the last ball touch the boundary of $E_0$,
	\item they overlap as little as possible, i.e. the distance between the centers of two neighbouring balls is $(1-\sigma_1/m_1)/(m_1-1)$.
\end{itemize}
We call these balls the \emph{balls of generation 1}, we denote their family by $\mathcal{B}_1$, and we set $E_1=\bigcup_{B\in \cB_1}B$.

Now suppose that $E_k$ has already been defined as a union of balls $\bigcup_{B\in \cB_k}B$, and that $\#\cB_k = m(k)$. Inside every ball $B\in\cB_k$ we place $m_{k+1}$ closed balls of diameter $2r_{k+1}:=\sigma_1\dots\sigma_{k+1}/m(k+1)$. We do it in such a way, that 
\begin{itemize}
	\item their centers lie on the diameter of $B$ forming angle $\sum_{i=1}^{k+1}\alpha_i$ with the $x$ axis,
	\item the boundaries of the first and the last ball touch touch the boundary of $B$,
	\item they overlap as little as possible, i.e. the distance between the centers of two neighbouring balls is
	\begin{equation*}
		d_{k+1} := \frac{\sigma_1\dots\sigma_k}{m(k)}\cdot \frac{1-\sigma_{k+1}/m_{k+1}}{m_{k+1}-1}.
	\end{equation*}
\end{itemize}
The balls defined above are called the \emph{balls of generation (k+1)}, and their family is denoted by $\cB_{k+1}$. Clearly, $\#\cB_{k+1} = m_{k+1}\cdot m(k) = m(k+1)$. We set $E_{k+1} = \bigcup_{B\in \cB_{k+1}}B$, and $E = \bigcap_{k\ge 0} E_k$.

It is shown in \cite[\S 2.1]{joyce2000set} that if $M$ is chosen appropriately, then two balls of generation $(k+1)$ may intersect only if they are contained in the same ball of generation $k$. It follows that there exists a natural probability measure $\mu$ supported on $E$ defined by
\begin{equation}\label{eq:def of mu}
	\mu(B) = m(k)^{-1}\quad \text{for}\ B\in\mathcal{B}_k,\ k\ge 1.
\end{equation}

If the sequence $\beta_k$ is chosen properly, the set $E$ has the following curious property: it is of non-$\sigma$-finite length, but all the projections of $E$ onto lines are of zero length. Moreover, the Menger curvature of $E$ is finite. However, we will not use those properties.
\subsection{BPBE(2) holds} In \cite[\S 2.1]{joyce2000set} Joyce and M\"{o}rters construct a function $\varphi:[0,d_1)\to \R$ satisfying $\varphi(r)<r$ and
\begin{equation*}
	\int_0^{d_1} \frac{\varphi(r)^2}{r^3}\ dr <\infty.
\end{equation*}
They also show that for $0<r<d_1$ the measure $\mu$ satisfies $\mu(B(x,r))\le 84\, \varphi(r)$. It follows easily that $\mu(B(x,r))\le C_1 r$ for $C_1=\max(84, 1/d_1)$ and all $r>0$, and so $\mu$ satisfies \eqref{eq:growth condition}. Furthermore, by the observations above and the fact that $\mu(\R^2)=1$, for all $x\in E=\supp\mu$ we have
\begin{equation}\label{eq:density sqared bdd}
	\int_0^{\infty} \left(\frac{\mu(B(x,r))}{r}\right)^2\ \frac{dr}{r} \lesssim \int_0^{d_1} \frac{\varphi(r)^2}{r^3}\ dr + \int_{d_1}^{\infty}\frac{1}{r^3}\ dr\le M_0
\end{equation}
for some $M_0$ depending only on $d_1$ and $\varphi$. 

Obviously, for any $\theta\in[0,\pi/2),\ \alpha\in (0,1), R>0,$ we have
\begin{equation*}
	\mathcal{E}_{\mu,2}(x,\theta,\alpha,R) = \int_0^{R} \left(\frac{\mu(K(x,\theta,\alpha,r))}{r}\right)^2\ \frac{dr}{r} \le \int_0^{\infty} \left(\frac{\mu(B(x,r))}{r}\right)^2\ \frac{dr}{r},
\end{equation*}
and so the BPBE(2) condition is trivially satisfied.

Let us note that the boundedness of nice singular integral operators on $L^2(\mu)$ for this particular measure $\mu$ is not a new result. It is well known that measures satisfying \eqref{eq:density sqared bdd} behave well with respect to SIOs. For example, one can use \eqref{eq:density sqared bdd} and \cite[Theorem 2.2]{mattila1996analytic} to prove local curvature condition for $\mu$, and then boundedness of Cauchy transform follows from \cite[Theorem 1.1]{tolsa1999}.

\subsection{\texorpdfstring{$\mathcal{E}_{\mu,1}$}{E\_\{mu,1\}} is not bounded}
Let $x\in E,\ \theta\in[0,\pi/2),$ and $\alpha\in (0,\pi/100)$ be given. We will show that
\begin{equation}\label{eq:E1 infinite}
	\mathcal{E}_{\mu,1}(x,\theta,\alpha,1) = \int_0^1 \frac{\mu(K(x,\theta,\alpha,r))}{r}\ \frac{dr}{r} = \infty.
\end{equation}

\begin{definition}
	We will say that an integer $k$ is a \emph{good index} if
	\begin{equation}\label{eq:good index 2}
		\bigg\lvert\big(\sum_{j=1}^k\alpha_j - N\pi\big) - \theta\bigg\rvert \le \frac{\alpha}{16},
	\end{equation}
	where $N$ is the integer satisfying $2^N\le k<2^{N+1}$ .
	By the definition of $\alpha_j$, this is equivalent to
	\begin{equation}\label{eq:good index}
		\big\lvert(k-2^N+1)\frac{\pi}{2^N} - \theta\big\rvert\le \frac{\alpha}{16}.
	\end{equation}
\end{definition}
Our strategy is the following: first, we show that there are many good indices. Then, we prove that if $k$ is a good index, then ${\mu(K(x,\theta,\alpha,2\,r_k))}{r_k^{-1}}$ is large. Put together, the two facts will imply \eqref{eq:E1 infinite}.

We define $N_0=N_0(\alpha)$ to be a large integer, to be fixed in Lemmas \ref{lem:many good indices} and \ref{lem:measure bigger than Bk}.
\begin{lemma}\label{lem:many good indices}
	If $N_0=N_0(\alpha)$ is large enough, then for all $N>N_0$ we have a large portion of good indices satisfying $2^N\le k< 2^{N+1}$, that is,
	\begin{equation*}
		{\#\{2^N\le k< 2^{N+1}\ :\ k\ \text{is a good index}\}}\gtrsim 2^N\alpha.
	\end{equation*}
\end{lemma}
\begin{proof}
	Let $N_0$ be so big that $2^{-N_0}\pi < \alpha/100$, and let $N>N_0$. Let $2^N\le k_0< 2^{N+1}$ be the index minimizing $\lvert(k_0-2^N+1)\pi 2^{-N} - \theta\big\rvert$. It is clear that
	\begin{equation*}
		\lvert(k_0-2^N+1) 2^{-N}\pi - \theta\big\rvert\le 2^{-N}\pi,
	\end{equation*}
	and so it follows from \eqref{eq:good index} that all integers $k$ such that $2^N\le k< 2^{N+1}$ and $|(k-k_0) 2^{-N}\pi| \le \alpha/50$ are good indices. It is easy to see that there are at least $C 2^N\alpha$ such integers, where $C$ is some absolute constant.
\end{proof}
Recall that $r_k$ was the radius of balls of $k$-th generation, and $x\in E$ is arbitrary. For $k\ge 1$ let $B_{k}\in\mathcal{B}_k$ be a ball of generation $k$ containing $x$ (there may be two such balls, in which case we
just choose one).
\begin{lemma}\label{lem:measure bigger than Bk}
	If $N_0=N_0(\alpha)$ is large enough, then for all good indices $k\ge 2^{N_0}$ we have
	\begin{equation*}
		{\mu(K(x,\theta,\alpha,2\,r_k))} \gtrsim \mu(B_k).
	\end{equation*}
\end{lemma}
\begin{proof}
	Let $y$ be the center of $B_{k+1}$, so that $|x-y|\le r_{k+1}$. By construction,
	\begin{equation}\label{eq:comparing rk 2}
		r_{k+1} = r_k\,\sigma_{k+1} (Mk)^{-1}\le r_k\, k^{-1}.
	\end{equation}
	Since $k\ge 2^{N_0}$, for $N_0$ big enough we get 
	\begin{equation}\label{eq:comparing rk}
		|x-y|\le r_{k+1}\le \frac{\sin({\alpha}/{50})r_k}{2}.
	\end{equation}
	Then, it follows from \lemref{lem:truncated cone included in another} that
	\begin{equation*}
		K(y,\theta,\alpha/8, \sin(\alpha/8)^{-1}|x-y|,r_k)\subset K(x,\theta,\alpha,2\, r_k).
	\end{equation*}
	Since $\sin({\alpha}/{8})^{-1}|x-y|\le r_k/2$ by \eqref{eq:comparing rk}, we get
	\begin{equation}\label{eq:cone inclusion 1}
		K(y,\theta,\alpha/8, r_k/2, r_k)\subset K(x,\theta,\alpha,2\, r_k).
	\end{equation}
	On the other hand, using the definition of good index \eqref{eq:good index 2} we arrive at
	\begin{equation}\label{eq:cone inclusion 2}
		K(y,\textstyle\sum_{j=1}^k\alpha_j - N\pi,\alpha/50, r_k/2, r_k)\subset K(y,\theta,\alpha/8, r_k/2, r_k).
	\end{equation}
	For brevity, set $\mathsf{K}$ to be the cone from the left hand side above, and let $L$ be the axis of $\mathsf{K}$. Recall that the diameter of $B_k$ (let us call it $D$) forms angle $\sum_{j=1}^k\alpha_j - N\pi$ with the $x$ axis; that is, $D$ is parallel to $L$. Since $y$ is the center of $B_{k+1}$, it follows from the construction of $E$ that $y\in D$. Hence, $D\subset L$.
	
	We claim that the balls of generation $(k+1)$ contained in $B_k\cap B(y,r_k)\setminus B(y,r_k/2)$, are in fact contained in $\mathsf{K}$. Indeed, suppose $z$ belongs to such ball, so that
	\begin{equation*}
		\dist(z,L)=\dist(z,D)\le r_{k+1} \overset{\eqref{eq:comparing rk}}{\le} \frac{\sin({\alpha}/{50})r_k}{2}r_k\le \sin({\alpha}/{50})|z-y|.
	\end{equation*}
	Thus, $z\in\mathsf{K}$. 
	
	Since $y\in D$ and $B_k$ is a ball of radius $r_k$, it follows that a large portion of balls of generation $(k+1)$ contained in $B_k$ is also contained in $B(y,r_k)\setminus B(y,r_k/2)$. That is, they are of the type considered above. Hence,
	\begin{equation*}
		\mu(\mathsf{K})\gtrsim \mu(B_k).
	\end{equation*}
	By \eqref{eq:cone inclusion 2} and \eqref{eq:cone inclusion 1} we have $\mathsf{K}\subset K(x,\theta,\alpha,2\, r_k)$, and so the proof is finished.
\end{proof}
\begin{lemma}\label{lem:large measure of Bk}
	For $k\ge 2$
	\begin{equation*}
		\frac{\mu(B_k)}{2\, r_k}\gtrsim \frac{1}{k}.
	\end{equation*}
\end{lemma}
\begin{proof}
	By the definition of $\mu$ \eqref{eq:def of mu}, $r_k,$ and $\sigma_k$ we have 
	\begin{multline*}
		\frac{\mu(B_k)}{2\, r_k} = m(k)^{-1}\,\frac{m(k)}{\sigma_1\dots\sigma_{k}} = \frac{1}{\sigma_1\dots\sigma_{k}} =
		\left(\frac{1}{2}\right)^{\beta_1}\dots 
		\bigg(\frac{k}{k+1}\bigg)^{\beta_k} \\
		\ge \frac{1}{2}\dots \frac{k}{k+1} = \frac{1}{k+1},
	\end{multline*}
	where in the last inequality we used the fact that $1/2< \beta_k< 1$.
\end{proof}
We are ready to finish the proof of the estimate \eqref{eq:E1 infinite}.
\begin{proof}[Proof of \eqref{eq:E1 infinite}]
	Observe that if $k>{N_0}$ is a good index, then by \lemref{lem:measure bigger than Bk} and \lemref{lem:large measure of Bk} for $r\in (2\,r_k, 4\,r_k)$
	\begin{equation*}
		\frac{\mu(K(x,\theta,\alpha,r))}{r}\gtrsim \frac{1}{k},
	\end{equation*}
	and so
	\begin{equation}\label{eq:big integral}
		\int_{2\,r_k}^{4\,r_k}\frac{\mu(K(x,\theta,\alpha,r))}{r} \frac{dr}{r} \gtrsim \frac{1}{k}.
	\end{equation}
	Recall that $r_{k+1}\le k^{-1}\,r_k$ by \eqref{eq:comparing rk 2}. Hence,
	\begin{multline*}
		\int_0^1 \frac{\mu(K(x,\theta,\alpha,r))}{r}\ \frac{dr}{r}\ge \sum_{k\ge 2^{N_0}}\int_{2\,r_k}^{4\,r_k} \frac{\mu(K(x,\theta,\alpha,r))}{r}\ \frac{dr}{r}\\
		\ge \sum_{N=N_0}^{\infty}\ \sum_{\substack{2^{N}\le k <2^{N+1}\\ k\, \text{is good}}}\int_{2\,r_k}^{4\,r_k} \frac{\mu(K(x,\theta,\alpha,r))}{r}\ \frac{dr}{r} \overset{\eqref{eq:big integral}}{\gtrsim}\sum_{N=N_0}^{\infty}\ \sum_{\substack{2^{N}\le k <2^{N+1}\\ k\, \text{is good}}} \frac{1}{k}\\
		\approx \sum_{N=N_0}^{\infty}\ \sum_{\substack{2^{N}\le k <2^{N+1}\\ k\, \text{is good}}} 2^{-N} \overset{\lemref{lem:many good indices}}{\gtrsim} \sum_{N=N_0}^{\infty} 2^{-N}2^N\alpha = \infty.
	\end{multline*}
\end{proof}

\end{document}